\documentclass{amsart}
\usepackage{amsmath,amssymb,amsthm,amsfonts}

\title{On the Hodge-Newton filtration for $p$-divisible $\mathcal{O}$-modules}
\author{Elena Mantovan}
\author{Eva Viehmann}
\address{Department of Mathematics, Caltech, Pasadena, CA 91125, USA}
\address{Mathematisches Institut der Universit\"{a}t Bonn, Beringstra\ss e 1, 53115 Bonn, Germany}

\newcommand{\BT}{$p$-divisible }
\newcommand{\NP}{Newton polygon }
\newcommand{\HP}{Hodge polygon }
\newcommand{\SHP}{$\sigma$-invariant Hodge polygon }

\newcommand{\NPs}{Newton polygons }

\newcommand{\RZ}{Rapoport-Zink }
\newcommand{\HN}{Hodge-Newton }
\newcommand{\Spec}{{\rm Spec} }
\newcommand{\Spf}{{\rm Spf} }
\newcommand{\End}{{\rm End} }

\newcommand{\MM}{{\mathcal M}}
\newcommand{\zz}{\mathbb Z}
\newcommand{\La}{\Lambda}
\newcommand{\hh}{\mathcal H}
\newcommand{\oo}{\mathcal O}
\newcommand{\ra}{\rightarrow}

\newcommand{\uh}{{(H,\iota)}}
\newcommand{\Res}{{\rm Res}}
\newcommand{\Fil}{{\rm Fil}}

\newcommand{\unu}{{\overline{\mu}}}

\newcommand{\Aa}{\mathfrak{a}}

\newtheorem{thm}{Theorem}
\newtheorem{thm*}{Theorem}
\newtheorem{defn}[thm]{Definition}
\newtheorem{defn*}[thm*]{Definition}
\newtheorem{lemma}[thm]{Lemma}
\newtheorem{lemma*}[thm*]{Lemma}
\newtheorem{cor}[thm]{Corollary}

 \theoremstyle{remark}
\newtheorem{remark}[thm]{Remark}

\begin{document}

\begin{abstract}
The notions \HN decomposition and \HN filtration for $F$-crystals
are due to Katz and generalize Messing's result on
the existence of the local-\'etale filtration for \BT groups.
Recently, some of Katz's classical results have been generalized
by Kottwitz to the context of $F$-crystals with additional structures
and by Moonen to $\mu$-ordinary \BT  groups.

In this paper, we discuss further generalizations to the situation
of crystals in characteristic $p$ and of \BT groups with additional structure by endomorphisms.
\end{abstract}
\maketitle

\section{Introduction}
This paper is concerned with crystals in characteristic $p$ and with
\BT groups with given endomorphisms. We begin by defining these
notions. Let $p$ be a prime. Let $B$ be a split unramified
finite-dimensional semi-simple algebra over $\mathbb{Q}_p$ and let
$\oo_B$ be a maximal order of $B$. Let $A$ be a noetherian ring and
a formally smooth $\mathbb{F}_p$-algebra. Assume that there is an
ideal $\mathfrak{I}$ of $A$ such that $A$ is $\mathfrak{I}$-adically
complete and such that $A/\mathfrak{I}$ is a finitely generated
algebra over a field with a finite $p$-basis. In \cite{dJ}, 2.2.3,
de Jong shows that there is an equivalence of categories between
crystals over a formal scheme $S=\Spf (A)$ as above and the category
of locally free $\tilde{A}$-modules together with a connection. Here
$\tilde{A}$ denotes a lift of $A$ over $\mathbb{Z}_p$ in the sense
of \cite{dJ}, 1.2.2 and 1.3.3, i.e. a $p$-adically complete flat
$\mathbb{Z}_p$-algebra endowed with an isomorphism
$\tilde{A}/p\tilde{A}\cong A$.  We write $\sigma$ for a lift of
Frobenius on $\tilde{A}$.  In the following definition we implicitly
use this equivalence but consider an additional $\oo_B$-action.

\begin{defn}\label{def1}
Let $S$ be as above and $a\in \mathbb{N}\setminus \{0\}$. A
$\sigma^a$-$F$-crystal of $\oo_B$-modules is a $4-$tuple
$(M,\nabla,F,\iota)$ consisting of
\begin{itemize}
\item a locally free $\tilde{A}$-module $M$ of finite rank,
\item an integrable nilpotent $W(k)$-connection $\nabla$,
\item a horizontal morphism $F:(\sigma^a)^*(M,\nabla)\rightarrow (M,\nabla)$ which induces an isomorphism after inverting $p$ and
\item a faithful action $\iota:\oo_B\rightarrow \End(M)$ commuting with $\nabla$ and $\sigma^a$-commuting with $F$.
\end{itemize}
\end{defn}

Note that when $A$ is a perfect field, the connection can be reconstructed from $(M,F)$, and thus will often be omitted from the notation. An important class of examples of $\sigma$-$F$-crystals of $\oo_B$-modules is provided by the crystals associated via (contravariant) Dieudonn\'{e} theory to $p$-divisible groups with endomorphisms. Let $S$ be a formal scheme which is locally of the form $\Spf(A)$ for $A$ as above. Then by \cite{dJ}, Main Theorem 1, the crystalline Dieudonn\'{e} functor is an equivalence of categories between \BT groups over $S$ and their Dieudonn\'{e} crystals. For later use we define \BT $\oo_B$-modules in a slightly more general context.
\begin{defn}
For any $\mathbb{Z}_p$-scheme $S$, by a \BT $\oo_B$-module $\uh$ over $S$ we mean a \BT group $H$ over
$S$, together with a faithful action $\iota:\oo_B\hookrightarrow
{\rm End }(H)$.
\end{defn}

\begin{remark}
A standard reduction argument shows that to prove any of our results on $\sigma$-$F$-crystals of $\oo_B$-modules or on \BT $\oo_B$-modules, we may assume that $B$ is
simple, i.e.~a matrix algebra over an unramified extension $D$ of
$\mathbb{Q}_p$ (compare \cite{RZ}, Section 3.23(b)). Then by Morita
equivalence we may further assume $B=D$.
\end{remark}

We now define Newton and Hodge polygons of $\sigma^a$-$F$-crystals of $\oo_B$-modules in the case that $A=k$ is an algebraically closed field of positive characteristic
$p$. In this case $\tilde{A}=W(k)$. Let $(M,\nabla,F,\iota)$ be a $\sigma^a$-$F$-crystal of $\oo_B$-modules over $\Spec(k)$ where $M$ is a $W(k)$-module of rank $h$. The \NP $\nu$ of $(M,\nabla,F,\iota)$ is defined as the \NP of the map
$F$. It is an element of $(\mathbb{Q}_{\geq 0})^h_+$, i.e. an
$h$-tuple of nonnegative rational numbers that are ordered
increasingly by size. The notion polygon is used as one often considers the polygon associated to $\nu=(\nu_i)$ which is the graph of the continuous, piecewise linear function $[0,h]\rightarrow \mathbb{R}$ mapping $0$ to $0$ and with slope $\nu_i$ on $[i-1,i]$. We will refer to this polygon by the same letter $\nu$, for example when we talk about points lying on $\nu$, or about the slopes of $\nu$. If $(M,\nabla,F,\iota)$ is the $\sigma$-$F$-crystal of $\oo_B$-modules associated to a \BT $\oo_B$-module $(H,\iota)$ over $k$ then $\nu_i\in[0,1]$ for all $i$. Note that the isogeny class of $H$ is uniquely
determined by $\nu$.

A second invariant of $(M,\nabla,F,\iota)$ is its Hodge polygon
$\mu\in \mathbb{N}^h_+$. It is given by the condition that the relative
position of $M$ and $F(M)$ is $\mu$. Then $\nu\preceq\mu$ with respect to the
usual order, i.e. if we denote the entries of $\nu$ and $\mu$ by
$\nu_i$ and $\mu_i$, then $\sum_{i=1}^l (\nu_i-\mu_i)\geq 0$ for all
$l\leq h$ with equality for $l=h$. If $(M,\nabla,F,\iota)$ is the Dieudonn\'{e} module of some $p$-divisible $\oo_B$-module, then the entries of $\mu$ are all $0$ or $1$.

For a $\sigma^a$-$F$-crystal of $\oo_B$-modules $(M,\nabla,F,\iota)$ over an algebraically closed field $k$, we analyse how the Galois group ${\rm
Gal}(\mathbb{F}_q\mid\mathbb{F}_p)$ (where $\mathbb{F}_q$ with
$q=p^r$ is the residue field of $B$) permutes the entries of $\nu$ and $\mu$. In Section \ref{sec2} we show that the entries of the \NP are fixed
by this action, however in general the entries of the \HP are not. We call
$$\overline{\mu}=\frac{1}{r}\sum_{i=0}^{r-1} \sigma^i(\mu)$$ the
{\em $\sigma$-invariant Hodge polygon} of $(M,\nabla,F,\iota)$. Here $\sigma$ is the Frobenius. Then we
also have $\nu\preceq\unu$.

If $(M,\nabla,F,\iota)$ is a $\sigma^a$-$F$-crystal of
$\oo_B$-modules over a scheme $S$, the Newton polygon (resp. Hodge
polygon, $\sigma$-invariant Hodge polygon) of $(M,\nabla,F,\iota)$
is defined to be the function assigning to each geometric point $s$
of $S$ the Newton polygon (resp. Hodge polygon, $\sigma$-invariant
Hodge polygon) of the fiber of $(M,\nabla,F,\iota)$ at $s$.

In the literature the $\sigma$-invariant \HP of a $p$-divisible $\oo_B$-module is often called
the $\mu$-ordinary polygon, because of its analogy with the ordinary
polygon in the classical context (e.g. see \cite{Wedhorn}, 2.3).
Indeed, for \BT $\oo_B$-modules $\unu$ can be defined as follows. For a given $\mu\in
\{0,1\}^h_+$, one can consider the set $X$ of all isogeny classes of \BT
$\oo_B$-modules with \HP equal to $\mu$. Their \NPs all lie above
$\mu$ and share the same start and end points. Thus, there is a
partial order on $X$ given by the natural partial order $\preceq$ on
the set of Newton polygons. The $\mu$-ordinary polygon $\unu$ is the unique
maximal element in $X$. In the classical case without endomorphisms the
$\mu$-ordinary polygon coincides with the Hodge polygon $\mu$.

\begin{defn} \label{def2}
Let $(M,\nabla,F,\iota)$ be a $\sigma^a$-$F$-crystal of $\oo_B$-modules over $S$. Let $(x_1,x_2)=x\in\mathbb{Z}^2$ be a point which lies on the Newton
polygon $\nu$ of $(M,\nabla,F,\iota)$ at every geometric point of
$S$. Let $\nu_1$ denote the polygon consisting of the first $x_1$
slopes of $\nu$ and $\nu_2$ the polygon consisting of the remaining
ones. Let $\unu_1$ and $\unu_2$ be the analogous parts of the $\sigma$-invariant Hodge polygon.
\begin{enumerate}
\item $x$ is called a breakpoint of $\nu$ if the slopes of $\nu_1$ are strictly smaller than the slopes
of $\nu_2$.
\item We say that $(M,\nabla,F,\iota)$ has a \HN decomposition at $x$ if there are $\sigma^a$-$F$-subcrystals of $\oo_B$-modules
$(M_1,\nabla|_{M_1},F|_{M_1},\iota|_{M_1})$ and $(M_2,\nabla|_{M_2}, F|_{M_2}$, $\iota|_{M_2})$ of  $(M,\nabla,F,\iota)$ with $M_1\oplus M_2=M$ and such that the Newton polygon of $M_i$ is $\nu_i$ and its \SHP is $\unu_i$.
\item We say that $(M,\nabla,F,\iota)$ has a \HN filtration in $x$ if there is a $\sigma^a$-$F$-subcrystal of $\oo_B$-modules
$(M_1,\nabla|_{M_1},F|_{M_1},\iota|_{M_1})$ of $(M,\nabla,F,\iota)$ with Newton polygon $\nu_1$ and \SHP $\unu_1$ and such that
$M/M_1$ is a $\sigma^a$-$F$-crystal of $\oo_B$-modules with Newton polygon $\nu_2$ and \SHP $\unu_2$.
\end{enumerate}
We will use the analogous notions for \BT $\oo_B$-modules, also more generally over $\mathbb{Z}_p$-schemes.
\end{defn}

\begin{remark}\label{rem1}
Note that the existence of a \HN decomposition or filtration in some $x$ as in the previous definition immediately implies that $x$ lies on $\unu$ at every geometric point of $S$.
\end{remark}

In this paper we give conditions for the existence of a \HN decomposition over a perfect
field of characteristic $p$, a \HN filtration for families in
characteristic $p$, and a \HN filtration for deformations of \BT $\oo_B$-modules to
characteristic 0. First instances of these results can be found in
Messing's thesis \cite{ME} in the classical context of \BT groups
and in Katz's paper \cite{KZ} for $F$-crystals in positive
characteristics. Let us now explain the results in more detail.

\subsection{The \HN decomposition for $\oo_B$-modules}
Our result on the \HN decomposition for \BT $\oo_B$-modules over a
perfect field of characteristic $p$ is closely related to a more
general result for $F$-crystals with additional structures. The
generalization of Katz's result (which deals with the cases of
$GL_n$ and $GSp_{2n}$) to the context of unramified reductive groups is
due to Kottwitz, in \cite{KO5}. (To be precise the
condition required in \cite{KO5} is actually stronger than the one we give,
on the other hand Kottwitz's argument for its sufficiency applies almost
unchanged to our settings.)

The following notations are the same as in loc. cit.

Let $F$ be a finite unramified extension of $\mathbb{Q}_p$ or
$\mathbb{F}_p((t))$ and $\varepsilon\in\{p,t\}$ a uniformizer. Let
$L$ be the completion of the maximal unramified extension of $F$ in
some algebraic closure. We write $\mathcal{O}_F$ and $\mathcal{O}_L$
for the valuation rings, and $\sigma$ for the Frobenius of $L$ over
$F$.

Let $G$ be an unramified reductive group over $\mathcal{O}_F$ and
$T$ a maximal torus of $G$ over $\mathcal{O}_F$. Let $P_0$ be a
Borel subgroup containing $T$ with unipotent radical $U$.

Let $P=MN$ be a parabolic subgroup of $G$ containing $P_0$ with Levi
component $M$ containing $T$ and assume that all of these groups are
defined over $\mathcal{O}_F$. Let $A_P$ be the maximal split torus
in the center of $M$ and write
$\Aa_P=X_*(A_P)\otimes_{\zz}\mathbb{R}$. For $P=P_0$ we skip the
index $P$. Let $X_M$ be the quotient of $X_*(T)$ by the coroot
lattice for $M$. Then the Frobenius $\sigma$ acts on $X_M$ and we
denote by $Y_M$ the $\sigma$-coinvariants. We identify
$Y_M\otimes_{\zz}\mathbb{R}$ with $\Aa_P$. Let $Y_M^+$ be the subset of $Y_M$
of elements identified with elements $x\in\Aa_P$ with $\langle
x,\alpha\rangle>0$ for each root $\alpha$ of $A_P$ in $N$.

In \cite{KO4}, Kottwitz defines a morphism $w_G:G(L)\rightarrow X_G$, which induces a map $\kappa_G$
assigning to each $\sigma$-conjugacy class of elements of $G(L)$ an element of $Y_G$. The classification
of $\sigma$-conjugacy classes in \cite{KO3} and \cite{KO4} shows that a $\sigma$-conjugacy class of some
$b\in G(L)$ is determined by $\kappa_G(b)$ and the Newton point $\nu$ of $b$, a dominant element of
$X_*(A)_{\mathbb{Q}}=X_*(A)\otimes_{\mathbb{Z}}\mathbb{Q}$. For two elements $\nu_1,\nu_2\in X_*(A)_{\mathbb{Q}}$ we write $\nu_1\preceq_G \nu_2$ if $\nu_2-\nu_1$ is a non-negative rational linear combination of positive coroots of $G$.

Let $K=G(\oo_L)$. Let $b\in G(L)$ and $\mu\in X_*(T)$ dominant. The
{\em affine Deligne-Lusztig set associated to $b$ and $\mu$} is the
set
$$X_{\mu}(b)=\{g\in G(L)/K\mid g^{-1}b\sigma(g)\in K\varepsilon^{\mu}K\}.$$

In the function field case, the affine Deligne-Lusztig set is the
set of $\overline{\mathbb{F}}_p$-valued points of a subscheme of the
affine Grassmannian. In the case of mixed characteristic, a scheme
structure is not known in general. However, for $G$ the restriction of
scalars of some $GL_n$ or $GSp_{2n}$ and  $\mu$ minuscule,
$X_\mu(b)$ is in bijection with the $\overline{\mathbb{F}}_p$-valued
points of a moduli space of \BT $\oo_B$-modules constructed by
Rapoport and Zink, \cite{RZ}. In the case of mixed
characteristic, for $G$ the restriction of scalars of some $GL_n$ or
$GSp_{2n}$ (and any choice of $b,\mu$), to every element $g\in
X_\mu(b)$ we can associate a crystal with given endomorphisms and/or
polarization. Let $G$ be the restriction of scalars of $GL_B(V)$ or
$GSp_B(V)$, for $B$ an unramified finite extension of $F$  and $V$ a
finite dimensional $B$-vector space. Let $\oo_B$ be the valuation
ring of $B$. We choose an $\oo_B$-lattice $\Lambda$ in $V$. In the symplectic case, we assume
$B$ is endowed with an involution $*$ of the first kind which preserves
$\oo_B$, and $V$ with a non-degenerate alternating $*$-hermitian
form with respect to which $\La$ is self\-dual. Then, to any element
$g\in G(L)$ we associate the $\sigma$-$F$-crystal of $\oo_B$-modules $(\La_{\oo_L},\Phi)$, where
$\La_{\oo_L}=\La\otimes_{\oo_F} \oo_L$ and
$\Phi=g^{-1}b\sigma(g)\circ (id_\La\otimes \sigma)$. It is naturally
endowed with a faithful action of $\oo_B$ and in the case of
$G=GSp_B(V)$  also a polarization. The isomorphism class of
$(\La_{\oo_L},\Phi)$ is determined by the image of $g$ in
$G(L)/K$, while its isogeny class (i.e. the associated isocrystal)
only depends on $b$. For $g\in X_\mu(b)$, the corresponding module
has \HP $\mu$, and for $\mu$ minuscule it is the
crystal of a \BT $\oo_B$-module.

\begin{thm}\label{thmkottwitz}\label{thm1}
Let $\mu_M\in X_*(T)$ be $M$-dominant and let $b\in M(L)$ such that
$\kappa_M(b)= \mu_M$ and $\nu\preceq_M\mu_M$. Then
$X^M_{\mu_M}(b)\neq\emptyset$. Let $\mu_G$ be the $G$-dominant
element in the Weyl group orbit of $\mu_M$. If $\mu_M=\mu_G$ and $\kappa_M(b)\in Y_M^+$, then the natural
inclusion $X^M_{\mu_M}(b)\hookrightarrow X^G_{\mu_G}(b)$ is a
bijection.
\end{thm}

For $F$-crystals without additional structures, i.e. for $GL_n$ and
$GSp_{2n}$, this is Katz's Theorem 1.6.1. in \cite{KZ}. If $F$ is an extension of $\mathbb{Q}_p$ and if $b$ is basic in $M$, this is shown by Kottwitz in \cite{KO5}. If $G$ is split, Theorem \ref{thm1} is the same as \cite{conncomp}, Thm. 1.

Applying it to the contravariant Dieudonn\'{e} module of \BT $\oo_B$-modules, this theorem yields the existence statement of the following corollary, see Section \ref{sec3}. The general assertion of the corollary is a special case of Corollary \ref{cor10}.
\begin{cor}\label{cor1}
Let $(H,\iota)$ be a \BT $\oo_B$-module over a perfect
field of characteristic $p$, with \NP $\nu$ and \SHP
$\overline{\mu}$. Let $(x_1,x_2)=x\in\mathbb{Z}^2$ be on $\nu$. If
$x$ is a breakpoint of $\nu$ and $x$ lies on $\overline{\mu}$, then
$(H,\iota)$ has a unique \HN decomposition associated to $x$.
\end{cor}

In the classical context of $p$-divisible groups without additional
structures, the \HN decomposition coincides with the
multiplicative-bilocal-\'etale decomposition. Indeed, this is a simple consequence of the fact that
the slopes of the \NP of a \BT group are bounded by $0$ and $1$,
while the slopes of its Hodge polygon are always either $0$ or $1$. Its existence and
uniqueness are shown by Messing in \cite{ME}.

\subsection{The \HN filtration for families in characteristic $p$.}
This result on the existence of a \HN filtration for families of
$\oo_B$-modules over a smooth scheme of positive characteristic is a generalization of Katz's Theorem 2.4.2 in \cite{KZ}.

\begin{thm}\label{thm2}
Let $S=\Spf(A)$ be a formal scheme as in Definition \ref{def1}. Let $(M,\nabla, F,\iota)$ be a $\sigma$-$F$-crystal of $\oo_B$-modules over $S$ with \NP $\nu$ and \SHP $\overline{\mu}$. Let $(x_1,x_2)=x\in\mathbb{Z}^2$ be a breakpoint of $\nu$ lying on $\unu$ at
every geometric point of $S$. Then there is a \HN filtration of
$(M,\nabla, F,\iota)$ in $x$ and it is unique.

Furthermore, if $S=Spec (A)$ for $A$ a perfect
$\mathbb{F}_p$-algebra, then the \HN filtration admits a unique splitting.
\end{thm}

Theorem \ref{thm2} applied to the crystal of a \BT $\oo_B$-module yields the following corollary.

\begin{cor}\label{cor10}
Let $(H,\iota)$ be a \BT $\oo_B$-module over a formal scheme $S$ which is locally of the form $\Spf (A)$ for some $A$ as in Definition \ref{def1}. Let $\nu$ be the \NP of $(H,\iota)$ and $\unu$ its \SHP. Assume that there is a breakpoint $(x_1,x_2)=x\in\mathbb{Z}^2$ of $\nu$ lying on $\unu$ for every geometric point of $S$. Then there is a unique \HN filtration of
$(H,\iota)$ in $x$. If $S=\Spec (A)$ for $A$ a perfect
$\mathbb{F}_p$-algebra, then the \HN filtration admits a unique splitting.
\end{cor}

For \BT groups (without additional structures) the \HN
filtration coincides with the multiplicative-bilocal-\'etale
filtration. Therefore, in this context, Katz's result was also known at the
time due to the work of Messing in \cite{ME}.

\subsection{The \HN filtration for deformations of \BT $\oo_B$-modules to characteristic 0.}

Let $k$ be a perfect field of positive characteristic $p$ and $R$ an
artinian local ring with residue field $k$. By {\em a deformation}
of a \BT $\oo_B$-module $(H,\iota)$ over $k$ to $R$ we
mean a \BT $\oo_B$-module $(\hh,\iota)$ over $R$
together with an isomorphism $j:(H,\iota)\ra (\hh,\iota)\times_{\Spec(R)}\Spec (k)$.

\begin{thm}\label{thm3}
Let $R$ be an artinian local $\zz_p$-algebra of residue field $k$ of
characteristic $p$. Let $(\hh,\iota)$ be a \BT $\oo_B$-module over
$R$ and let $H$ be its reduction over $k$. Let $\nu$ and $\overline{\mu}$ be the \NP and the \SHP of $(H,\iota)$. Assume that there is a
\HN decomposition of $H$ associated to a breakpoint
$x\in\mathbb{Z}^2$ of $\nu$.
Let $H_2\subset H$ be the induced filtration. If $\nu_1=\unu_1$, then the filtration of
$(H,\iota)$ lifts (in a unique way) to a \HN filtration of $\hh$ in
$x$.
\end{thm}

For $p$-divisible groups without $\oo_B$-module structure,
if a breakpoint of the \NP lies on the Hodge polygon, then the two
polygons necessarily share a side (of slope either $0$ or $1$).
Therefore, in that case, the condition in this theorem coincides
with the one in Theorem \ref{thm2}.

In \cite{ME} Messing establishes that the infinitesimal deformations of a \BT
group over a perfect field of characteristic $p$ are naturally
endowed with a unique filtration lifting the
multiplicative-bilocale-\'etale decomposition. In the case of \BT
groups with additional structures, the existence of a \HN filtration
for deformations was previously observed in a special case in
the work of Moonen. More precisely, in \cite{BEN}, Moonen
establishes the existence of the slope decomposition for
$\mu$-ordinary \BT $\oo_B$-modules over perfect fields and of the
slope filtration for their deformations. Since, by definition, in
these cases the slope decomposition agrees with the (finest) \HN
decomposition, Moonen's results can be regarded as special cases of the existence of the \HN decomposition and filtration
for \BT $\oo_B$-modules. For the proof of Theorem \ref{thm3},
we use a generalization of Moonen's argument via crystalline Dieudonn\'e theory.

\begin{remark}\label{rem3}
There is also a variant of a Hodge-Newton decomposition and filtration for polarized \BT $\oo_B$-modules. For this case we assume that $p$ is odd, and that $B$ is endowed with an involution $*$ that preserves $\oo_B$ (see
\cite{KO2}, Section 2). A polarized \BT $\oo_B$-module is a \BT $\oo_B$-module $(H,\iota)$ together with a polarization $\lambda:H\rightarrow H^{\vee}$ satisfying the condition $$\lambda\circ b^*=b^{\vee}\circ \lambda \text{ for all } b\in \oo_B.$$
Let $(H,\iota)$ be a polarized \BT $\oo_B$-module of dimension $n$. From the existence of the polarization
and the uniqueness of the \HN decomposition and filtration one obtains the following results.
\begin{enumerate}\item If $x=(x_1,x_2)$ is a breakpoint of the Newton point of $H$, then $x'=(2n-x_1,n-x_2+x_1)$ also is.
\item If $(H,\iota)$ has a \HN decomposition $H=H_1\oplus H_2$ associated to $x$ as in Theorem \ref{thmkottwitz}, then it also has a \HN decomposition in $x'$. We may assume $x_1\leq n$. Then there is a decomposition $H\cong H^{\vee}=H_1\oplus H'\oplus H_1^{\vee}$ of $\oo_B$-modules such that $ H_1^{\vee}\oplus H'=H_2$ and that $(H_1\oplus H')\oplus H_1^{\vee}$ is the \HN decomposition in $x'$. Here $H'$ is trivial if $x_1=n$.
\item If $(H,\iota)$ has a \HN filtration $H_2\subseteq H$ associated to $x$ as in Theorem \ref{thm2} or \ref{thm3}, then it also has a \HN filtration $H_2'\subseteq H$ in $x'$. Again we assume $x_1\leq n$. If $x_1=n$, then the two filtrations coincide. In general there is a joint filtration $H_2'\subseteq H_2\subseteq H$ of $\oo_B$-modules such that $H_2^{\vee}\cong H/H_2'$ and $(H_2')^{\vee}\cong H/H_2$ where the isomorphisms are given by the principal polarization on $H$.
\end{enumerate}
A detailed discussion of the notion of Hodge-Newton filtration for
polarized $F$-crystals without endomorphisms can be found in \cite{csima}.
\end{remark}

The paper is organized as follows. In Section \ref{sec2} we consider $\sigma$-$F$-crystals of $\oo_B$-modules with additional
structure in more detail. In Section \ref{sec3} we prove Theorem
\ref{thmkottwitz} and explain the relation between the general
result and the special case of \BT $\oo_B$-modules. The
last two sections concern the \HN filtration for families of $\sigma$-$F$-crystals of $\oo_B$-modules in
characteristic $p$ and for deformations (in the \BT case) to characteristic 0,
respectively.\\

Our results on the existence of \HN filtrations are used by the
first author in \cite{MAN}. There, the existence of a \HN filtration
allows to compare the corresponding Rapoport-Zink space with a
similar moduli space of filtered $p$-divisible groups with
endomorphisms. This leads to a proof of some cases of a conjecture
of Harris (\cite{HA}, Conjecture 5.2), namely that in the
cases when a \HN filtration exists, the $l$-adic cohomology of the
corresponding moduli space is parabolically induced from that of a
lower-dimensional one.

\subsection{Acknowledgments}
We thank L. Fargues, M. Harris and R. Kottwitz for their interest in
our work and for stimulating discussions. We are grateful to M. Rapoport and T. Wedhorn for helpful comments on a preliminary version of this paper. This project started while
the authors were visiting the Institut des Hautes \'Etudes
Scientifiques (E.M.) and the Universit\'e de Paris-Sud (E.V.). We
thank both institutions for their hospitality. During the stay at Orsay, E.V. was supported by a fellowship within the Post-Doc program of the German Academic Exchange Service (DAAD).

\section{$\sigma$-$F$-crystals with additional structures}
\label{sec2}\label{secfurther}

Let $(M,\nabla,F,\iota)$ be a $\sigma$-$F$-crystal of $\oo_B$-modules over a scheme $S$ over an algebraically closed field $k$ as in Definition \ref{def1}. Recall that we assume that $B$ is an
unramified field extension of $\mathbb{Q}_{p}$ of degree $r$.
We consider $(M,\nabla,F^r,\iota)$. One easily sees that the connection is again compatible with $F^r$ and $\iota$, thus it is a $\sigma^{r}$-$F$-crystal of $\oo_B$-modules. Let $I={\rm~Hom~}(\oo_B,W(\overline{\mathbb{F}}_p))$. Note that $M$ is a module over $\oo_B\otimes_{\mathbb{Z}_p}\tilde{A}$ which (as $k$ contains $\overline{\mathbb{F}}_{p}$) is a product of $r$ copies of $\tilde{A}$. Thus the module $M$ decomposes naturally into a direct sum associated to the characters,
$$M=\oplus_{i\in I} M^{(i)}.$$ We identify the set
$I$ with $\zz/r\zz$, with $r$ being the degree of the extension of
$k$ given by the residue field of $\oo_B$. To do so we fix $j\in I$
and consider the bijection that to each $s\in \zz/r\zz$ associates
the element $\sigma^{s}\circ j$ in $I$. As $\nabla$ and $F^r$ commute with the $\oo_B$-action, they decompose into a direct sum of connections resp. $\sigma^r$-linear morphisms on the different summands. We obtain a decomposition of the $\sigma^r$-$F$-crystal of $\oo_B$-modules $(M,\nabla,F^r,\iota)$ into $r$ summands $(M^{(i)},\nabla|_{M^{(i)}},F^r|_{M^{(i)}})$.

To compute \NPs and Hodge polygons, we assume for the remainder of this section that $(M,\nabla,F^r,\iota)$ is defined over an algebraically closed field of characteristic $p$. For crystals over a scheme, one computes the polygons at each geometric point separately.

The following notations are the same as in \cite{BEN},
Section 1.

For each $i$, the Frobenius
$F:M\ra M$ restricts to a $\sigma$-linear map $F_i:M^{(i)}\ra M^{(i+1)}$.
In particular, $F^r|_{M^{(i)}}=\phi_i$ is given by $\phi_i=F_{i+r-1}\circ F_{i+r-2}\circ \cdots \circ F_i.$
Moreover, for all $i\in I$, the morphism
$F_i:M^{(i)}\ra M^{(i+1)}$ is a semilinear morphism of
$\sigma^r$-$F$-crystals, thus the $M^{(i)}$ are all isogenous. This implies, in
particular, that they have the same Newton polygon.

The \NP of $(M,F)$ and the \NP of $(M^{(i)},\phi_i)$ can easily be
calculated from one another. Indeed, the $M^{(i)}$ are mutually isogenous summands of $(M,F^r)$, thus a slope $\lambda$ appears in the \NP of $(M,F)$ with
multiplicity $m$ if and only if the slope $r\lambda$ appears in the
\NP of each $(M^{(i)},\phi_i)$ with multiplicity $m/r$. We denote by
$\nu'$ the \NP of $(M^{(i)},\phi_i)$  and call it the $r$-reduction
of $\nu$. These calculations also show that $\nu$ is already
invariant under the Galois group (as mentioned in the introduction).

The next step is to compute the $\sigma$-invariant \HP $\unu$ more explicitly, see also \cite{Wedhorn}, 2.3. It is defined as $\frac{1}{r}\sum_{i\in \mathbb{Z}/r\mathbb{Z}}\sigma^i(\mu)$, where $\mu$ is the Hodge polygon. We consider the decomposition
$$M/FM=\oplus_i M^{(i)}/F_{i-1}M^{(i-1)}$$ and let $$d=h/r=\dim_k M^{(i)}/pM^{(i)}.$$ Let $m_i\in\mathbb{Z}^d_+$ be the relative position of $M^{(i)}$ and $F_{i-1}M^{(i-1)}$. Recall that the \NP of $(M^{(i)},\phi_i)$ is $\nu'$ for all $i$. Let $\mu^{(i)}$ be the \HP of $(M^{(i)},\phi_i)$.

\begin{remark}\label{remunu}
We can rearrange the entries of $\mu$ into $r$ blocks of length $d$ and such that the $i$th block is $m_i$ and corresponds to the contribution by $M^{(i)}$. Then $\sigma(\mu)$ is obtained from $\mu$ by replacing each block by the following one. This implies that each entry of $\unu$ occurs with a multiplicity divisible by $r$. If we use the
same formal reduction procedure to define the $r$-reduction $\unu'$ of
$\unu$ we obtain $\unu'=\sum_i m_i$. However, $\unu'$ is in general not equal to the Hodge
polygons of the $(M^{(i)},\phi_i)$ (which also do not need to be
equal), compare Lemma \ref{lemcompmunu}.
\end{remark}

Note that in the special case where $\mu$ is minuscule (for example for the crystal associated to a \BT group), $m_i$ can be written as $(1,\dotsc,1,0,\dotsc,0)$ with multiplicities $f(i)$ and $d-f(i)$. Here $$f(i)=\dim_k M^{(i)}/F_{i-1}M^{(i-1)}$$ with $0\leq f(i)\leq d,i=1,\dots ,m$. Note that $f$ need not be constant. The data $(d, f)$ is called the {\em type} of $(M,F)$. This description together with Remark \ref{remunu} implies the following result.

\begin{cor}\label{propben}
(\cite{BEN}, Lemma 1.3.4)
Let $(M,F)$ be the crystal associated to a \BT group over an algebraically closed field of characteristic $p$. The entries of $\unu'$ are $0\leq a_1\leq a_2\leq
\cdots \leq a_d$, where
$$a_j=\#\{ i\in I | \, f(i)>d-j\}.$$
In particular, the number of breakpoints of $\unu$ is equal to the
number of distinct values of the function $f$.
\end{cor}

\begin{remark}
Let $(M,F)$ be a (not necessarily minuscule) $\sigma^r$-$F$-crystal over a perfect field and let
$F':M\rightarrow M$ be $\sigma$-linear. Let $\mu_F, \mu_{F'},
\mu_{F\circ F'}$ denote the Hodge polygons of $F, F',$ and $F\circ F'$. Then
$\mu_{F\circ F'}\preceq \mu_F+\mu_{F'}$. Indeed, if we denote their
slopes by $\mu_{F,i}, \mu_{F',i}, \mu_{F\circ F',i}$, we have to
show that $\sum_{i=1}^l (\mu_{F,i}+ \mu_{F',i}- \mu_{F\circ
F',i})\leq 0$ for all $l\geq 1$. Considering the $l$th exterior
power of $M$, this is equivalent to showing that the least Hodge
slope of $(\wedge^l F)\circ (\wedge^l F')$ is greater than or
equal to the sum of the least Hodge slopes of $\wedge^l F$ and
$\wedge^l F'$. But this is obvious as the least Hodge slope of one
of these morphisms is the minimal $i$ such that the image of $M$
under this morphism is not contained in $p^{i+1}M$.
\end{remark}

\begin{lemma}\label{lemcompmunu}
In the situation above we have
$$\nu'\preceq \mu^{(i)}\preceq \unu'$$ for all $i\in\{1,\dotsc,r\}$.
\end{lemma}
\begin{proof}
For all $i\in\mathbb{Z}/r\mathbb{Z}$, we choose an isomorphism between $M^{(i)}$ and $W(k)^d$. Then each restriction $F_i:M^{(i)}\rightarrow M^{(i+1)}$ of $F$ induces a $\sigma$-linear morphism $W(k)^d\rightarrow W(k)^d$ which we also denote by $F_i$. Its Hodge polygon is $m_i$ (where we use the notation from Remark \ref{remunu}). Then $\phi_i$ is identified with $F_{i+r-1}\circ\dotsm\circ F_i$ where the indices are still taken modulo $r$. Its Newton polygon is $\nu'$ and its Hodge polygon $\mu^{(i)}$, thus we obtain the first of the two inequalities. The remark above shows that the Hodge polygon of $\phi_i=F_{i+r-1}\circ\dotsm\circ F_i$ is less or equal to the sum of the Hodge polygons of $F_{i+r-1},\dotsc, F_i$. Hence $\mu^{(i)}\preceq \sum_{j\in I} m_j=\unu'$.
\end{proof}

\begin{remark}\label{rem12}
Lemma \ref{lemcompmunu} implies that if $\nu'$ and $\unu'$ have a point in common, then
all the Hodge polygons $\mu^{(i)}$ also contain this point.
\end{remark}

\section{\HN decompositions for \BT groups with additional structure}\label{sec3}

\begin{proof}[Proof of Theorem \ref{thmkottwitz}]
Under the additional assumptions that $F$ is an extension of $\mathbb{Q}_p$ and that $b$ is basic in $M$, the theorem is \cite{KO5}, Theorem 4.1. But Kottwitz's proof still works without these additional assumptions. To leave out the assumption on the characteristic, the proof may remain unmodified. Instead of the assumption on $b$ to be basic, Kottwitz's proof only uses in the proof of \cite{KO5}, Lemma 3.2 that $M$ contains the centralizer $M_b$ of $\nu$. But this already follows from $\nu\in\Aa_P^+$.
\end{proof}

To relate this theorem to \cite{KZ}, Theorem 1.6.1, we consider the special case $G=GL_n$. Let $T$ be the diagonal torus and $P_0$ the Borel subgroup of lower triangular matrices. Then $X_*(T)\cong \zz^{n}$ and an element $\mu=(\mu_1,\dotsc,\mu_n)$ is dominant if $\mu_1\leq \dotsm\leq \mu_n.$ Let $b\in G(L)$. Then its Newton point $\nu=(\nu_i)\in X_*(T)_{\mathbb{Q}}\cong \mathbb{Q}^{n}$ is the same as the $n$-tuple of Newton slopes considered in Katz's paper. The map $\kappa$ maps $b$ to $\sum_i \nu_i=m=v_p({\rm det}~b)\in\mathbb{Z}\cong \pi_1(G)$. The Hodge polygon $\mu$ is of the form $(0,\dotsc,0,1,\dotsc,1)$ with multiplicities $h-m$ and $m$. The ordering $\preceq$ takes the following explicit form. The simple coroots of $G$ are the $e_i-e_{i-1}$ for $i=2,\dotsc,n$. Thus $\nu=(\nu_i)\preceq \mu=(\mu_i)$ if and only if $$\sum_{i=1}^l (\nu_i-\mu_i)\geq 0$$ for all $l\in \{1,\dotsc,n\}$ and with equality for $l=n$. As $G$ is split, $X=Y$. Let $P\subseteq G$ be a maximal parabolic containing $P_0$ and let $M$ be its Levi factor containing $T$. Then $M\cong GL_j\times GL_{n-j}$ for some $j$. The condition $\nu\in Y_M^+$ is equivalent to $\nu_{a}-\nu_b>0$ for all $a> j$ and $b\leq j$. Thus $j$ is a breakpoint of the associated Newton polygon considered by Katz. We have $$\pi_1(M)=\zz^n/\{(x_i)\mid \sum_{i=1}^j x_i=0=\sum_{i=j+1}^n x_i\}.$$ Hence $\kappa_M(b)=\mu$ if and only if $\sum_{i=1}^j \nu_i=\sum_{i=1}^j \mu_i$ and $\sum_{i=j+1}^n \nu_i=\sum_{i=j+1}^n \mu_i$. This means that the polygons associated to $\mu$ and $\nu$ have the point $(j,\sum_{i=1}^j \nu_i)$ in common. From these considerations we see that in this special case Theorem \ref{thmkottwitz} reduces to Katz's theorem on the Hodge-Newton decomposition.

A similar translation shows the existence statement of Corollary \ref{cor1} over an algebraically closed field. Indeed, let $G={\rm Res}_{B\mid \mathbb{Q}_p}(GL_{d})$. Then $X_*(A)_{\mathbb{Q}}$ can be written as $\{(x_{ij})\in \mathbb{Q}^h\mid 1\leq i\leq r, 1\leq j\leq d\}$ and the action of $\sigma$ by mapping $(x_{ij})$ to $(y_{ij})$ with $y_{ij}=x_{i-1,j}$. We can identify the coinvariants under this action of the Galois group with the subset of $(x_{ij})\in X_*(A)_{\mathbb{Q}}$ with $x_{ij}=x_{i+1,j}$ for all $i$ and $j$, or, using the $r$-reduction, with $\mathbb{Q}^{d}$. Thus the Newton polygon of $H$, considered as an element of $Y_G$ is identified with $\nu'$ and its Hodge polygon is identified with $\unu'\in \mathbb{Q}^{d}$. As for the split case above one easily sees that the condition $b\in Y_M^+$ for $M={\Res}_{B\mid \mathbb{Q}_p}(GL_j\times GL_{h-j})$ is equivalent to $\nu'$ having a breakpoint at $j$. Also, $\kappa_M(b)=\unu'$ translates into the condition that the two polygons coincide in this point. Thus Theorem \ref{thmkottwitz} implies the existence statement in the corollary for \BT $\oo_B$-modules over an algebraically closed field.

Similarly, for $G=GSp_{2n}$ or $\Res_{B|\mathbb{Q}_p}GSp_{2d}$, one may show that in this case,
Theorem \ref{thmkottwitz} implies the polarized variant of Katz's
Theorem or of Corollary \ref{cor1}. This is a second strategy to obtain the results explained in Remark \ref{rem3}.

\section{\HN filtrations in characteristic $p$}\label{sec4}

\begin{proof}[Proof of Theorem \ref{thm2}]

Our proof consists in reducing to Katz's result by deducing the statement
from the existence of a compatible system of \HN decompositions for the
$\sigma^r$-$F$-crystals $(M^{(i)}, \nabla|_{M^{(i)}},\phi_i)$, for all $i$. Here the $M^{(i)}$ are the same as in Section \ref{secfurther}. To be able to use them, we first consider the case that $k$ is algebraically closed and that $S$ is affine.

Let $x$ be the break-point of $\nu'$ corresponding to the breakpoint of $\nu$ in the statement. Then
it follows from Remark \ref{rem12} that $x$ also lies on $\mu^{(i)}$ for all $i$.
Therefore, by \cite{KZ}, Theorem 1.6.1, for each $i\in I$,
there is a unique filtration $(M^{(i)}_{1},\nabla|_{M^{(i)}_{1}},\phi_i|_{M^{(i)}_{1}})\subset(M^{(i)},\nabla|_{M^{(i)}},\phi_i)$ associated to $x$ as in Definition \ref{def2} and with the desired splitting property. Note that we have slightly less assumptions on $S$ than Katz, but the proof of his theorem still holds in this situation.

By construction, for all $i$, the sets of Newton slopes of
$M^{(i)}_{1}$ and $M^{(i)}/M^{(i)}_{1}$ are distinct,
therefore the map $F_i:M^{(i)}\ra M^{(i+1)}$ maps $M^{(i)}_{1}$ into $M^{(i+1)}_{1}$. This implies that $(\oplus_i M^{(i)}_{1},\nabla|_{\oplus_i M^{(i)}_{1}},F|_{\oplus_i M^{(i)}_{1}},\iota|_{\oplus_i M^{(i)}_{1}})$ inherits a natural structure of sub-$\sigma$-$F$-crystal of $\oo_B$-modules of $(M,\nabla,F,\iota)$ associated to the breakpoint of $\nu$ as in the statement.

Finally, if $k$ is not algebraically closed, we first get an analogous filtration over $S\times_{\Spec(k)}\Spec(\overline{k})$ where $\overline{k}$ is an algebraic closure of $k$. The uniqueness of the filtration implies that it is stable by the Galois group ${\rm Gal}(\overline{k}| k)$, and thus it descends to a filtration over $S$ itself. A similar argument extends the result to non-affine $S$.
\end{proof}

\section{Lifting the \HN filtration}\label{sec5}\nopagebreak
\subsection{Explicit coordinates for $\mu$-ordinary $p$-divisible groups}\label{sec51}\nopagebreak

Let $(H,\iota)$ be a \BT $\oo_B$-module over an algebraically closed field $k$ of characteristic $p$. Assume that $x=(d',n')$ is the first breakpoint of its \NP $\nu'$, that $x$ also lies on the \HP $\unu'$, and that $\nu'$ and $\unu'$ coincide up to this breakpoint. Then by Corollary \ref{cor1}, $(H,\iota)$ has a \HN decomposition in $x$, so its Dieudonn\'{e} module decomposes as $M=M_1\oplus M_2$. Here $M_1$ corresponds to the first parts of $\nu'$ and $\unu'$, thus it is $\mu$-ordinary and its polygons are constant. Let $\lambda=l/r$ with $0\leq l\leq r$ be their slope. From \cite{BEN}, 1.2.3 we obtain that $M_1$ has generators $e_{i,j}$ with $i\in\mathbb{Z}/r\mathbb{Z}$, $1\leq j \leq d'$ and such that $F(e_{i,j})=p^{a(i)}e_{i+1,j}$ where $a(i)\in \{0,1\}$ is equal to $1$ on exactly $l$ elements $i$. As the \NP and the \HP of $M$ both have slope $l/r$ up to $x$, we have that $f(i)=d$ if $a(i)=1$. (Of course, $f(i)\leq d-d'$ if $a(i)=0$.) For the whole module $M$, we may complement the coordinates $e_{i,j}$ to obtain $e_{i,j}$ with $1\leq j\leq d$ of each
  $M^{(i)}$ such that $e_{i,j}$ is in the image of $F$ if and only if $j> d-f(i)$.

\subsection{Description of the universal deformation}

We recall the description of the universal deformation via crystalline Dieudonn\'{e} theory from \cite{BEN}, 2.1., see also \cite{F}, 7. Let $A$ be a formally smooth $W(k)$-algebra with $k$ algebraically closed and $\varphi_A$ a lift of the Frobenius of $A/pA$ with $\varphi_A|_{W(k)}=\sigma$. Then \BT groups over $A$ can be described by 4-tuples $(\MM, \Fil^1(\MM),\nabla,F_{\MM})$ consisting of
\begin{itemize}
\item a free $A$-module $\MM$ of finite rank
\item a direct summand $\Fil^1(\MM)\subset \MM$
\item an integrable, topologically quasi-nilpotent connection $\nabla:\MM\rightarrow \MM\otimes \hat\Omega_{A/W(k)}$
\item a $\varphi_A$-linear endomorphism $F_{\MM}:\MM\rightarrow \MM$.
\end{itemize}

The 4-tuple corresponding to the universal deformation of a \BT $\oo_B$-module $(H,\iota)$ over $\Spec(k)$ can be described explicitly. Let $(M,F,\iota)$ be its Dieudonn\'{e} module. As in the previous section, one can find a basis $e_{i,j}$ with $1\leq j\leq d$ of each $M^{(i)}$ such that $e_{i,j}$ is in the image of $F$ if and only if $j> d-f(i)$. The Hodge filtration $\Fil^1(M)$ of $M$ is given by the $W(k)$-module generated by all $e_{i,j}$ with $j>d-f(i)$. The submodule $M^0$ generated by the $e_{i,j}$ with $j\leq d-f(i)$ is a complement. Let $A$ be the formal power series ring $$W(k)[[u^{(i)}_{l,m}\mid i\in \mathbb{Z}/r\mathbb{Z}, 1\leq m\leq d-f(i)<l\leq d]]$$ and let $g_{\rm univ}=\prod_i g^{(i)}\in GL_{\oo_B\otimes_{\mathbb{Z}_p}W(k)}(M)(A)=\prod_i GL_{d,W(k)}(A)$ be given by the $r$-tuple of matrices $g^{(i)}=(x_{a,b}^{(i)})_{1\leq a,b\leq d}$ with $$x_{a,b}^{(i)}=\begin{cases}
\delta_{a,b}&\text{if }a\leq d-f(i) \text{ or }b>d-f(i)\\
u^{(i)}_{a,b}&\text{otherwise.}
\end{cases}$$
Using this notation let $\MM=M\otimes_{W(k)}A$, $\Fil^1(\MM)=\Fil^1(M)\otimes_{W(k)}A$ and $F_{\MM}=g_{\rm univ}\circ (F_M\otimes \varphi_A).$ Then (see loc. cit.) there is a unique integrable topologically quasi-nilpotent connection $\nabla$ on $\MM$ that is compatible with $F_{\MM}$. Besides, $\oo_B$ acts by endomorphisms on $(\MM, \Fil^1(\MM),\nabla,F_{\MM})$ and this tuple corresponds to the universal deformation of the \BT $\oo_B$-module.

\begin{proof}[Proof of Theorem \ref{thm3}]

This proof is an adaptation of the corresponding proof for the $\mu$-ordinary case in \cite{BEN}, Proposition 2.1.9. Therefore we will not repeat each calculation but rather explain the differences. Again, Grothendieck-Messing deformation theory implies the uniqueness of the lifting. Besides it is enough to prove existence in the case that $k$ is algebraically closed, the generally case follows using descent and the uniqueness. Finally it is enough to prove the theorem for $\mathcal{H}$ the universal deformation of $H$.

We use the description of the universal deformation given above. Let $M_1\subset M$ be the submodule corresponding to the quotient $H_1$ of $H$. Then we have to show that $M_1$ lifts to a sub-4-tuple of $\MM$ with $\MM_1=M_1\otimes_{W(k)}A$. Let $d'=\min \{d-f(i)\mid i\leq r, f(i)<d\}$. The rank of $M_1$ is $rd'$ and $M_1=\oplus_{i}{\rm Span}\left(e_{i,j},1\leq j\leq d'\right)$. Denote by $M_1^{(i)}$ the parts of $M_1$ in the different $M^{(i)}$. From the explicit description of $M_1$ in Section \ref{sec51} we obtain that for each $i$ either $M_1^{(i)}\subseteq (M^0)^{(i)}$ or $(M^0)^{(i)}$ is trivial (the second case is equivalent to $f(i)=d$). In both cases $g^{\rm univ}$ maps $\MM_1$ to itself and hence $F_{\MM}$ restricts to $F_{\MM_1}=F_{M_1}\otimes \varphi_A$ on $\MM_1$.

It remains to prove that $\nabla(\MM_1)\subseteq \MM_1\otimes \hat\Omega_{A/W(k)}$, i.e. that the connection restricts to a connection on $\MM_1$. This calculation is the same as in Moonen's paper, see \cite{BEN},(2.1.9.1), at least under the condition $j<d'$. But this is enough as one only needs this equality for those $j$.
\end{proof}

Using duality, one easily obtains an analogous theorem in the case that $\nu$ and $\unu$ have the last breakpoint of $\nu$ in common and coincide from that point on.

Finally, we remark that the further assumption in this theorem, that
the \HP and \NP coincide up to the considered breakpoint or from this breakpoint on,
is necessary. Indeed, a first argument for this is already
visible in the proof: if this stronger condition is not satisfied,
then the key property that either $M_1^{(i)}\subseteq (M^0)^{(i)}$
or that $(M^0)^{(i)}$ is trivial for each $i$ does not hold. Thus there is some $i$ such that $g^{\rm univ}$ no longer
maps $M_1$ to itself, which makes a lifting of the filtration
impossible. Besides, the necessity of this condition can
also be seen independently of the proof by comparing the dimensions
of the corresponding period spaces.

\end{document}